\documentclass[12pt]{amsart}
\usepackage[top=30truemm,bottom=30truemm,left=30truemm,right=30truemm]{geometry}
\usepackage{mathrsfs}
\usepackage{bm}
\usepackage{amsfonts,amssymb}
\usepackage{dsfont}
\usepackage{extarrows}
\usepackage{amsmath}
\usepackage{mathrsfs}
\usepackage{enumerate}
\usepackage{amscd}
\usepackage[all,2cell]{xy}
\usepackage{hyperref}

\theoremstyle{plain}

\theoremstyle{definition}

\newtheorem{thm}{Theorem}[section]
\newtheorem{cor}[thm]{Corollary}
\newtheorem{lem}[thm]{Lemma}

\theoremstyle{definition}
\newtheorem{defn}{Definition}[section]

\theoremstyle{remark}
\newtheorem{rem}{\bf Remark}[section]

\newcommand{\be}{\begin{equation}}
\newcommand{\ee}{\end{equation}}
\newcommand{\bea}{\begin{eqnarray}}
\newcommand{\eea}{\end{eqnarray}}
\newcommand{\ben}{\begin{eqnarray*}}
	\newcommand{\een}{\end{eqnarray*}}
\newcommand{\bt}{\begin{split}}
	\newcommand{\et}{\end{split}}
\newcommand{\bet}{\begin{equation}}
\newcommand{\mc}{\mathbb{C}}

\newcommand{\ra}{\rightarrow}
\newcommand{\beq}{\begin{equation*}}
\newcommand{\eeq}{\end{equation*}}
\newcommand{\bal}{\begin{aligned}}
\newcommand{\eal}{\end{aligned}}

\newcommand{\ddbar}{\partial \bar{\partial}}
\newcommand{\dbar}{\bar{\partial}}

\newcommand{\calH}{{\mathcal{H}}}%
\newcommand{\calI}{{\mathcal{I}}}%
\newcommand{\calJ}{{\mathcal{J}}}%

\newcommand{\calO}{{\mathcal{O}}}%
\newcommand{\frakm}{\mathfrak{m}}%
\newcommand{\supp}{{\textup{Supp}\,}}%
\newcommand{\loc}{{\textup{loc}}}%
\newcommand{\Ker}{{\textup{Ker}}}

\newcommand{\Dom}{{\textup{Dom}}}

\newcommand{\pd}{\partial}%
\renewcommand{\leq}{\leqslant}%
\newcommand{\inner}[1]{\langle#1\rangle}
\newcommand{\iinner}[1]{\langle\langle#1\rangle\rangle}

\makeatletter
\g@addto@macro{\endabstract}{\@setabstract}
\newcommand{\authorfootnotes}{\renewcommand\thefootnote{\@fnsymbol\c@footnote}}%
\makeatother

\begin{document}
\title[Skoda's $L^2$ division]{Skoda's $L^2$ division theorem for $L^2$-optimal pairs}

\author[Z. Liu]{Zhuo Liu}
\address{Mathematical Science Research Center,  Chongqing University of Technology,  No. 69, Hongguang
Avenue, Banan District, Chongqing 400054, P. R. China.}
\email{liuzhuo@amss.ac.cn, liuzhuo@cqut.edu.cn}

\author[X.Zhang]{Xujun Zhang}
\address{School of Mathematics, Guangxi University, 100 Da Xue Dong Lu, Nanning, Guangxi, 530004, P. R. China.}
\address{Guangxi Center for Mathematical Research, Guangxi University, 100 Da Xue Dong Lu, Nanning, Guangxi, 530004, P. R. China.}
\email{xujunzhang@amss.ac.cn}





\maketitle

\begin{abstract}
We establish a Skoda-type $L^2$ division theorem for \(L^2\)-optimal pairs, using a technique that combines a new Bochner-type inequality derived from the \(L^2\)-optimal conditions and Skoda's basic inequality.
As applications, we provide some new characterizations of domains of holomorphy.
\end{abstract}


\section{Introduction}

The classical $L^2$ theory in several complex variables tells us that the curvature positivity of holomorphic vector bundle indicates the solvability for $\dbar$-equations with $L^2$ estimates,
including H\"ormander's $L^2$ existence, Skoda's $L^2$ division, the Ohsawa-Takegoshi $L^2$ extension and so on.
Recently in \cite{DNW21,DNWZ22,DWZZ}, Deng, Ning, Wang, Zhang and Zhou established the converse $L^2$ theory by giving alternative characterizations of plurisubharmonicity/Nakano positivity in terms of various $L^2$-conditions for $\bar{\pd}$.
They proved that a $C^2$-smooth real-valued function is plurisubharmonic if and only if it satisfies the ``\textit{optimal $L^2$-estimate condition}''.

Let us recall the following definition.
\begin{defn}[\cite{Liu-Zhang}]
Let $D$ be a domain in $\mc^n$ and $\varphi$ an upper semi-continuous function on $D$.
  We say that  a pair $(D,\varphi)$ is \emph{$L^2$-optimal}
   if for any smooth strictly plurisubharmonic function $\phi$ and
   any K\"{a}hler metric $\omega$ on $D$,
   the equation $\dbar u=f$ can be solved on $D$ for any $\dbar$-closed $(n,1)$-form
   $f \in L^2_{n,1}(D;loc)$
   with the estimate:
   \begin{equation}\label{eq:a1}
   \int_D|u|^2_{\omega,\varphi} e^{-\phi} dV_\omega
   \leq
   \int_D \inner{B_\phi^{-1}f,f}_{\omega,\varphi} e^{-\phi} dV_{\omega},
  \end{equation}
   provided that the right-hand side is finite, where $B_\phi:=[i\ddbar \phi, \Lambda_{\omega}]$.

   In particular, we call $D$ is  $\emph{$L^2$-optimal}$ if the pair $(D,0)$ is \emph{$L^2$-optimal}.
\end{defn}

\begin{rem}

It is worthwhile mentioning that the integrals on both sides of \eqref{eq:a1} are independent of the choice of the K\"ahler metric $\omega$ since $u$ is an $(n,0)$-form and $f$ is an $(n,1)$-form.
\end{rem}

In \cite{Deng-Zhang}, Deng-Zhang showed that
any bounded $L^2$-optimal domain in $\mc^n$ with $C^2$-smooth boundary is pseudoconvex.
Deng-Zhang's approach
is a combination of
Deng-Ning-Wang's characterization in \cite{DNW21} and the Morrey trick involving the boundary terms in the Bochner-H\"ormander-Kohn-Morrey formula (\cite{Hor65}).
However,
their technique fails when the $C^2$-smooth regularity condition is relaxed and the domain is unbounded.

Inspired by \cite{Pflug-1, Pflug,Diederich-Pflug},
in our previous work \cite{Liu-Zhang}, we showed that every bounded $L^2$-optimal domain with a null thin complement $( \textup{i.e.}\; \mathring{\overline{D}}=D)$ in a Stein manifold is also Stein via showing that the $L^2$-optimal condition implies a new variant of the celebrated  Skoda $L^2$ division theorem \cite{Skoda}.
However,
the division theorem we obtained in \cite{Liu-Zhang} is too special to deal with unbounded $L^2$-optimal  domains.

In this paper, by twisting the weight function with $\log|g|^2$ and the metric with $i\ddbar\log|g|^2$ simultaneously, we obtain a Bochner-type inequality on $D\setminus\{g=0\}$.
 Then, by reformulating Skoda's basic inequality, we obtain the classical Skoda $L^2$ division theorem for $L^2$-optimal pair $(D,\varphi)$ as follows.

\begin{thm}\label{thm:skoda-division'}
Let $D$ be a domain in $\mc^n$, $\varphi$  an upper semi-continuous function on $D$ and $g=(g_1,\cdots,g_p)\in\calO(D)^{\oplus p}$.
	Assume that the pair $(D,\varphi)$ is $L^2$-optimal. Set $\varepsilon>0$ and $m=\min\{n,p-1\} $. Then for any holomorphic  $(n,0)$-form $f$  with $$\int_Di^{n^2} f\wedge\bar fe^{-\varphi-(1+m+\varepsilon)\log|g|^2}<+\infty,$$ there exist holomorphic  $(n,0)$-forms $(h_1,\cdots,h_p)$ on $D$
such that
$\sum_{j=1}^{p} h_j g_j=f$
and
\begin{equation*}
  \sum_{j=1}^{p}\int_Di^{n^2}h_j\wedge\bar h_je^{-\varphi-(m+\varepsilon)\log|g|^{2}}
  \le(1+\frac{m}{\varepsilon})\int_Di^{n^2}f\wedge\bar fe^{-\varphi-(1+m+\varepsilon)\log|g|^2}.
\end{equation*}
\end{thm}

\begin{rem}
\begin{enumerate}
  \item In \cite{Liu-Zhang}, a division theorem was established under the assumption that $\varphi$ is continuous and that $0 < \inf_{z \in D} |g|^2 \le \sup_{z \in D} |g|^2 < +\infty$,  with coefficients depending on $\inf_{z \in D} |g|^2$.
By comparison, Theorem \ref{thm:skoda-division'} relaxes the continuity condition on $\varphi$ to upper semicontinuity and removes the requirement that $0 < \inf_{z \in D} |g|^2 \le \sup_{z \in D} |g|^2 < +\infty$. Moreover,  Theorem \ref{thm:skoda-division'} shares the same $L^2$ estimates as the celebrated Skoda $L^2$ division theorem in \cite{Skoda}, where the coefficient $1+\frac{m}{\varepsilon}$ is uniform and optimal in a certain sense.
  \item Recently,
Li-Meng-Ning-Zhou in \cite{Li-Meng-Ning-Zhou} gave a converse of Skoda's $L^2$ division.
\end{enumerate}

\end{rem}

As an application of Theorem \ref{thm:skoda-division'},
we can remove the assumption of boundedness for the domain when characterizing pseudoconvex domains via $L^2$-optimal conditions in \cite{Deng-Zhang,Liu-Zhang}.

\begin{cor}\label{cor1}
  Let $D\subset\mc^n$ be a domain with a null thin complement. Then  $D$ is  $L^2$-optimal if and only if $D$ is  a domain of holomorphy.
\end{cor}

 Moreover, we give a new characterization of domains of holomorphy via the $L^2$-optimal condition involving  the distance function.

\begin{cor}\label{cor2}
  Let $D$ be a domain in $\mc^n$. Then $D$ is a domain of holomorphy if and only if $(D,-\log d(z,D^c))$  is $L^2$-optimal.
\end{cor}

In addition, it is well-known that any Stein manifold admits a smooth strictly plurisubharmonic exhaustion function.
Similarly, we can obtain a Skoda-type $L^2$ division for $L^2$-optimal pairs on Stein manifolds.

\begin{thm}\label{thm:skoda-division-line-bundle}
Let $D$ be a domain in a Stein manifold $(X,\omega)$ and $L$ be a holomorphic line bundle endowed with a possibly singular metric $e^{-\varphi}$ over $D$ and $g=(g_{1},\cdots, g_{p}) \in \calO(D)^{\oplus p}$. Assume that the local weight function $\varphi$ is upper semi-continuous and  the pair $(D,L,\varphi)$ is $L^2$-optimal.
Set $\varepsilon>0$ and $ m=\min \{n,p-1 \}$.
Then for any $L$-valued $(n,0)$-form $f$ satisfying
$$
\int_D |f|^2_{\omega,\varphi} e^{-(1+m+\varepsilon)\log|g|^2} dV_{\omega}<+\infty,
$$
there exist holomorphic $L$-valued $(n,0)$-forms $(h_1,\cdots,h_p)$ on $D$
such that
$$
\sum_{j=1}^{p} h_j g_j=f
$$
and
\begin{equation}\label{for a}
  \sum_{j=1}^{p}\int_D |h_j|^2_{\omega,\varphi} e^{-(m+\varepsilon)\log|g|^{2}} dV_{\omega}
  \le(1+\frac{m}{\varepsilon})\int_D |f|^2_{\omega,\varphi}e^{-(1+m+\varepsilon)\log|g|^2}dV_{\omega}.
\end{equation}
\end{thm}

\begin{rem}
It is worth noting  that the integrals on both sides of \eqref{for a} are independent of the choice of the K\"ahler metric $\omega$ since $f$ and $h$ are  $(n,0)$-forms.
\end{rem}

For any holomorphic vector bundle over a Stein manifold,  we know that there exists a hypersurface such that the bundle is trivial outside the hypersurface. Therefore, similar to \cite{Siu},
we can translate Skoda's $L^2$ division into the following algebraic geometric formulation involving multiplier ideal sheaves.

\begin{cor}\label{thm:algebraic}
 Let $X$ be a Stein manifold and $L$  a holomorphic line bundle endowed with a possibly singular metric $e^{-\varphi}$ over $X$. Assume that the local weight function $\varphi$ is upper semi-continuous  and  the pair $(X,L,\varphi)$ is $L^2$-optimal. Let $M$ be an effective holomorphic line bundle on $X$ and $g=(g_{1},\cdots, g_{p}) \in (H^0(X,M))^{\oplus p}$.
Let $k  \ge 1$ be an integer and
define
	$$
	\calI_k =\calI (\varphi+(n+k)\log |g|^{2}),
	$$
then we have
\begin{equation*}
	H^0(X,K_X\otimes L\otimes (n+k+1)M\otimes\calI_{k+1})=\sum^{p}_{j=1} g^{j} H^0 (X,K_X\otimes L\otimes (n+k)M\otimes\calI_k).
\end{equation*}
\end{cor}

The rest of this paper is organized as follows.
In \S \ref{sec:notations},
we clarify some basic notations and prepare the key ingredients of the proof of Theorem \ref{thm:skoda-division'};
In \S \ref{sec:prove-main},
we prove Theorem \ref{thm:skoda-division'} and provide some generalizations;
In \S \ref{sec:application},
we present some applications of Theorem \ref{thm:skoda-division'}, including some new  characterizations of domain
of holomorphy via $L^2$-optimal conditions and
an algebraic geometric formulation of Theorem \ref{thm:skoda-division-line-bundle}.

\textbf{Acknowledgement:}
The authors are very grateful to Professor
Xiangyu Zhou for valuable suggestions on related topics.
Zhuo Liu was supported by the National Natural Science Foundation of China (No.12501101) and the Scientific Research Foundation of Chongqing University of Technology (No.2025ZDZ013).
Xujun Zhang was supported by the National Natural Science Foundation of China (No.12501102).
Xujun Zhang was also supported by the Tianyuan Mathematical Center in Southwest China, Guangxi Base.

\section{Basic estimate for Skoda's division}\label{sec:notations}

Let $D$ be a domain in $\mc^n$, $\omega$ be a K\"ahler metric on $D$ and $\varphi$ be an upper semi-continuous function on $D$ such that $e^{-\varphi}\in L^1(D;\loc)$.
Let $\wedge_c^{p,q}(D)$ be the space of all compactly supported smooth $(p,q)$-forms on $D$,
and $L_{p,q}^2(D;\loc)$ be the space of all  $(p,q)$-forms on $D$ whose coefficients are $L_\loc^2$ functions.

Let $|\cdot|^2_{\omega,\varphi}:=|\cdot|^2_{\omega}e^{-\varphi}$ and $\inner{\cdot,\cdot}_{\omega,\varphi}:=\inner{\cdot,\cdot}_{\omega}e^{-\varphi}$ denote the norm and the inner product on $\wedge^{p,q}T^*_D$ that induced by $\omega$ and $e^{-\varphi}$.
$dV_\omega:=\frac{\omega^n}{n!}$ denotes the volume form induced by $\omega$. Define
$$
	L_{p,q}^2(D;\omega,\varphi) := \left\{f\in L_{p,q}^2(D;\loc): \int_D |f|_{\omega,\varphi}^2dV_\omega <+\infty \right\},
$$
and
$$
\iinner{\cdot, \cdot }_{\omega, \varphi}= \int_{D}\inner{\cdot,\cdot}_{\omega,\varphi} dV_\omega.
$$
Then $L_{p,q}^2(D;\omega,\varphi) $ is a Hilbert space. Since $\varphi$ is  upper semi-continuous, we have
$L_{p,q}^2(D;\omega,\varphi)\subset L_{p,q}^2(D;\loc)$. Then we compute $\dbar$ in the sense of  distribution theory and we say that $f\in \Dom (\dbar)$ if $\dbar f\in L_{p,q+1}^2(D;\omega,\varphi)$. Since
$e^{-\varphi}\in L^1(D;\loc)$, we have $\wedge_c^{p,q}(D)\subset\Dom (\dbar) $. Therefore, the linear operator $\dbar: L_{p,q}^2(D;\omega,\varphi)  \rightarrow L_{p,q+1}^2(D;\omega,\varphi) $ is   closed  and densely defined.

Now we assume that $(D,\varphi)$ is $L^2$-optimal. Firstly, we observe that the multiplier ideal sheaf $\calI(\varphi)$ associated to $\varphi$ is coherent. Although the proof is attributed to \cite{Nadel, Dem-93JDG}, we include a detailed proof here for the sake of completeness.

\begin{lem}[{\cite{Nadel, Dem-93JDG}}]\label{lem coherence}
  Let $D$ be a domain in $\mc^n$ and $\varphi$ be an upper semi-continuous function on $D$. Define
   the multiplier ideal sheaf $\calI(\varphi)\subset\calO_D$ associated to $\varphi$ by $$\calI(\varphi)_x:=\{f\in\calO_{D,x}~:~|f|^2e^{-\varphi} ~\text{is}~ L^1~ \text{ near}~ x\}.$$
  Assume that $(D,\varphi)$ is $L^2$-optimal, then $\calI(\varphi)$   is coherent.
\end{lem}
\begin{proof} We use the argument of  \cite[Lemma 4.4]{Dem-93JDG}.

  Let $\calH( D,\varphi)$ be the Hilbert space of holomorphic functions $f\in \calO_D(D)$ satisfying $\int_D|f|^2e^{-\varphi}<+\infty$. By the strong Noetherian property of coherent sheaves, the family of sheaves generated by finite subsets of $\calH( D,\varphi)$ has a maximal element on each compact subset of $D$, hence $\calH( D,\varphi)$ generates a coherent ideal sheaf $\calJ\subset\calO_D $. It is clear that $\calJ\subset\calI(\varphi)$. 
Thanks to the Krull lemma, in order to prove the equality, we need only check that
$$ \calJ_x + \calI(\varphi)_x \cap  \frakm_{D,x}^{s+1}    = \calI(\varphi)_x $$
for every $x\in D$ and every integer $s$.

Let $f\in  \calI(\varphi)_x$ be defined in a neighborhood $V_x$ of
$x$ and let $\rho$ be a cut-off function with support in $V_x$ such that $\rho =1$ in a neighborhood $V'_x$ of
$x$. Denote $z=(z_1,\cdots,z_n)$ be the coordinate of $\mc^n$, $\omega=i\sum_{j=1}^{n}dz_j\wedge d\bar z_j$ and $dz=dz_1\wedge\cdots\wedge dz_n$.   Since $(D,\varphi)$ is $L^2$-optimal, we can solve the equation $\dbar u_\varepsilon=\dbar(\rho f)\wedge dz$  with respect to the smooth strictly plurisubharmonic weight
$$\phi_\varepsilon(z)=(n+s)\log(|z-x|^2+\varepsilon^2)+|z|^2.$$
Then we have \begin{align*}
               \int_D |u_\varepsilon|^2_{\omega,\varphi}e^{-\phi_\varepsilon} dV_\omega & \le\int_{D}\frac{|\dbar(\rho f)\wedge dz|_{\omega,\varphi}^2e^{ -|z|^2}}{(|z-x|^2+\varepsilon^2)^{n+s}}dV_\omega \\
                & \le C\int_{V_x\setminus \overline{V'_x}}|f|^2e^{-\varphi}|z-x|^{-2(n+s)}dV_\omega  <+\infty.
             \end{align*}
Since $\varphi$ is upper semi-continuous and $\phi_\varepsilon$ decreasingly converges to $\phi(z):=(n+s)\log|z-x|^2 +|z|^2$ as $\varepsilon\to0$,
  then by the Banach-Alaoglu-Bourbaki theorem, we can take a sequence $\varepsilon_j\to0$ such that $u_{\varepsilon_j}$ is weakly $L^2$ convergent to a limit $u$ on $D$. Then it follows from the weakly closedness of $\dbar$ that $\dbar u=\dbar(\rho f)\wedge dz$ on $D$. And by Fatou's lemma and the dominated convergence theorem, we have
 \begin{align*}
               \int_D |u |^2_{\omega,\varphi}e^{-\phi } dV_\omega   \le C\int_{V_x\setminus \overline{V'_x}}|f|^2e^{-\varphi}|z-x|^{-2(n+s)}dV_\omega  <+\infty.
             \end{align*}
 Thus $F=\rho f-u/dz$ is holomorphic on $D$, $F\in \calH( D,\varphi)$ and $f_x-F_x= u_x \in \calI(\varphi)_x \cap  \frakm_{D,x}^{s+1}$. This proves the coherence.
\end{proof}

\begin{rem}
  Let $D$ be a domain in a Stein manifold $X$ and $L$ be a holomorphic line
bundle endowed with a metric $e^{-\varphi}$ over $D$, where $\varphi$ is upper semi-continuous. If $(D,L,\varphi)$ is $L^2$-optimal in the sense of Definition \ref{def stein l2 optimal}, then one can show that $\calI(\varphi)$ is also coherent by modifying the proof and replacing  $\phi_\varepsilon(z)$ with $$\phi_\varepsilon(z)=(n+s)\rho\log(|z-x|^2+\varepsilon^2)+|\eta|^2,$$   where $\eta$ is a smooth, sufficiently strictly plurisubharmonic function on $X$.
\end{rem}

Let $g=(g_1, g_2,\cdots, g_p) \in \calO(D)^{\oplus p}$.
For any given holomorphic function $f$ on $D$,
the division problem is to find holomorphic functions $(h_1,\cdots,h_p)$ on $D$ such that
$$
T_1 h=\sum_{j=1}^{p} h_j g_j=f.
$$

Since $\calI(\varphi)$ is coherent, the support of $\calO_D/\calI(\varphi)$ is a complex analytic subset of $D$.
Set $D':=D\setminus(\supp\calO_D/\calI(\varphi)\cup\{g=0\})$, $\psi_{g,\gamma}:=\gamma|z|^2+\log|g|^2$ for some constant $\gamma>0$,
then $e^{-\varphi}$ is locally integrable on $D'$ and $\omega_{g,\gamma}:=i\partial\dbar \psi_{g,\gamma}$ defines a K\"ahler metric on $D'$.

For any $a>0$,
set $H_1=[L_{n,0}^2(D';\omega_{g,\gamma},\varphi+a\psi_{g,\gamma})]^{\oplus p}$
with inner product
$$
 \iinner{h, h' }_{\omega_{g,\gamma},\varphi+a\psi_{g,\gamma}} := \sum^{p} _{j=1} \iinner{ h_j, h'_j }_{\omega_{g,\gamma},\varphi+a\psi_{g,\gamma}}$$
and the corresponding norm
$$
	\|h \|^2_{\omega_{g,\gamma},\varphi+a\psi_{g,\gamma}}:=\iinner{ h, h } _{\omega_{g,\gamma},\varphi+a\psi_{g,\gamma}}
$$
for any $h=(h_1,\cdots,h_p), h'=(h'_1,\cdots,h'_p) \in H_1.$

Set $H_2=L^2_{n,0}(D';\omega_{g,\gamma},\varphi+a\psi_{g,\gamma}+\log|g|^2)$.
We define an operator $T_{1}$ as follows:
$$
\begin{aligned}
	T_{1}:H_1 &\ra H_2,
	\\
	h=(h_1,h_2,\cdots,h_p) & \ra \sum_{j=1}^{p} g_j h_j.
\end{aligned}
$$
Simple calculations show that $T_1$ is a continuous linear operator.

Set
$H_3=[L^2_{n,1} (D';\omega_{g,\gamma}, \varphi+a\psi_{g,\gamma})]^{\oplus p}$, and
we define a linear operator $T_2$ as follows:
$$
\begin{aligned}
	T_2:H_1
	   &\ra H_3,\\
	   h=(h_1,h_2,\cdots,h_p) &\ra
	    (\dbar h_1,\dbar h_2,\cdots,\dbar h_p).
\end{aligned}
$$
Then $T_2$ is closed and densely defined. 
Additionally, $T_1$ sends $\Ker (T_2)$ to the closed subspace
 $L^2_{n,0} (D';\omega_{g,\gamma}, \varphi+a\psi_{g,\gamma})\cap \Ker(\dbar) \subset H_2$.

Let us recall the following functional lemma which is owing to Skoda.

\begin{lem}[\cite{Skoda}]\label{lem:functional-lemma-Skoda}
	Let $H_1,H_2,H_3$ be Hilbert spaces,
    $T_1:H_1 \ra H_2$ be a continuous linear operator and
    $T_2:H_1 \ra H_3$ be closed and densely defined
	operator with $T_{1}(\Ker (T_2)) \subset F$,
where $F$ is a closed subspace of $H_2$.
	Then $T_1(\Ker(T_2))=F$ if and only if
	there exists a constant $c >0$
        such that
\begin{align*}
c\| u \|_{H_2} \leq\| T^*_{1} u+T^*_{2} v \|_{H_1}
\end{align*}
		holds for all $ u \in  F$
		and $ v \in \Dom(T^*_2)\cap (\Ker(T_2^*))^\bot$.
In this case, given $u\in F$, there exists $h\in\Ker(T_2)$ such that $T_1h=u$ and $\|h\|_{H_1}\le c^{-1}\|u\|_{H_2}$.
\end{lem}

Therefore, in order to solve the division problem, it suffices to
 verify the inequality
\begin{equation*}
	c \| u\|^2_{H_2}
\le  \|T_1^*u+T_2^*v\|_{H_1}^2
\end{equation*}
for any $u \in \Ker(\dbar)\cap H_2,
v \in  \Dom (T_2^*)\cap(\Ker(T_2^*))^\bot$ and
some constant $c>0$.

Since for any $h\in H_1$,
\begin{align*}
  \iinner{T_1 h,u}_{\omega_{g,\gamma},\varphi+a\psi_{g,\gamma}+\log|g|^2}&=\int_{D'} \sum_{j=1}^{p}\langle g_{j}h_{j}, u\rangle_{\omega_{g,\gamma}} e^{-\varphi-a\psi_{g,\gamma}-\log|g|^2}dV_{\omega_{g,\gamma}}\\
  &=\sum_{j=1}^{p}\int_{D'}\langle h_{j},\bar{g_{j}} u \rangle_{\omega_{g,\gamma}} e^{-\varphi-a\psi_{g,\gamma}-\log|g|^2}dV_{\omega_{g,\gamma}}\\
&=\sum_{j=1}^{p}\iinner{h_j,\bar{g_{j}}e^{-\log|g|^2} u}_{\omega,\varphi+a\psi_{g,\gamma}},
\end{align*}	
we get that for any $u\in H_2$,
\begin{align*}
	T_1^* u =(\bar{g}_1e^{-\log|g|^2} u,\bar{g}_2e^{-\log|g|^2} u ,\cdots,\bar{g}_pe^{-\log|g|^2} u).
\end{align*}
Then we have  \begin{align*}
    \| T_1^* u\|_{\omega_{g,\gamma},\varphi+a\psi_{g,\gamma}}^2=
    \|u\|_{\omega_{g,\gamma},\varphi+a\psi_{g,\gamma}+\log|g|^2}^2.
  \end{align*}
Hence
\begin{align}
    &\| T_1^* u+ T_2^ * v\|_{\omega_{g,\gamma},\varphi+a\psi_{g,\gamma}}^2 \nonumber \\=&\| T_1^* u\|_{\omega_{g,\gamma},\varphi+a\psi_{g,\gamma}}^2+\|  T_2^ * v\|_{\omega_{g,\gamma},\varphi+a\psi_{g,\gamma}}^2+2\text{Re}\langle T_1^* u, T_2^ *v\rangle_{\omega_{g,\gamma},\varphi+a\psi_{g,\gamma}} \nonumber \\ 
    =&\|u\|_{\omega_{g,\gamma},\varphi+a\psi_{g,\gamma}+\log|g|^2}^2+\|  T_2^ * v\|_{\omega_{g,\gamma},\varphi+a\psi_{g,\gamma}}^2+2\text{Re}\langle T_2(T_1^* u), v\rangle_{\omega_{g,\gamma},\varphi+a\psi_{g,\gamma}}.\label{eq:ineq-functional-1}
\end{align}

Now
we estimates the last two terms in \eqref{eq:ineq-functional-1} separately.

\subsection{A Bochner-type inequality from the $L^2$-optimal condition}

In order to estimate the term $\|  T_2^ * v\|_{\omega_{g,\gamma},\varphi+a\psi_{g,\gamma}}^2$, we need to show that we can solve the $\dbar$-equation on $D'$ with respect to $(\omega_{g,\gamma},\varphi+a\psi_{g,\gamma})$.

It is well known that any complex analytic subset is $L^2$-negligible in the sense of the following lemma.

\begin{lem}[{\cite[Chapter VIII-(7.3)]{Demailly}}]\label{lem L2negligible}
  Let $\Omega$ be an open subset of $\mc^n$ and $E$  a complex analytic subset of $\Omega$.
Assume that $v$ is  a $(p, q-1)$-form with $L^2_\loc$
coefficients and $w$ a $(p, q)$-form with $L^1_\loc$ coefficients
such that $\dbar v=w$ on $\Omega\backslash E$ (in the sense of distribution theory). Then $\dbar v=w$ on $\Omega$.
\end{lem}

\begin{lem}[{\cite[Chapter VI-(5.9)]{Demailly}}]
Let $(X, \omega)$ be a hermitian manifold. For simplicity, we denote by $\omega$  the operator defined by $\omega\wedge u$ for every  $u\in\wedge^{p,q}T_X^*$  and let $\Lambda_\omega$ denote its adjoint, then we have
$$[\omega,\Lambda_\omega]u=(p+q-n)u.$$
\end{lem}

Notice that  for every smooth strictly plurisubharmonic function $\phi$ on $D$, $i\ddbar \phi$ is a  K\"ahler metric on $D$. Then by the above lemma,  for every $u\in\wedge^{p,q}T_D^*$, we have 
$$
[i\ddbar \phi, \Lambda_{i\ddbar \phi}]u=(p+q-n)u.
$$
 Observing  this, if $(D,\varphi)$ is $L^2$-optimal, then
it is possible to solve a certain type of $\dbar$-equation on $D'$.

\begin{lem}\label{bbb}
Let $D$ be a domain in $\mc^n$ and $\varphi$ an upper semi-continuous function on $D$.
  Assume that $(D,\varphi)$ is $L^2$-optimal. Then for any $f\in L_{n,1}^2(D';\omega_{g,\gamma},\varphi+a\psi_{g,\gamma})\cap\Ker(\dbar)$, $a>0$, there is  a $u\in L_{n,0}^2(D';\omega_{g,\gamma},\varphi+a\psi_{g,\gamma})$ such that $\dbar u=f$ and
  \begin{equation*}
    \int_{D'}^{}|u|^2_{\omega_{g,\gamma}}e^{-\varphi-a\psi_{g,\gamma}}dV_{\omega_{g,\gamma}}\le\int_{D'}^{}\langle B_{a}^{-1}f,f\rangle_{\omega_{g,\gamma}}e^{-\varphi-a\psi_{g,\gamma}}dV_{\omega_{g,\gamma}},
  \end{equation*}
  provided that the right-hand side is finite, where $B_a:=B_{a\psi_{g,\gamma},\omega_{g,\gamma}}=[a i\ddbar\psi_{g,\gamma},\Lambda_{\omega_{g,\gamma}}]$.
\end{lem}
\begin{proof}
Since $\varphi$ is upper semi-continuous,
	then by Lemma \ref{lem L2negligible}, $f$ can be seen as a $\dbar$-closed $(n,1)$-form on  $D$.
  Let $g_\varepsilon:=(g_1, g_2,\cdots, g_p,\varepsilon) \in \calO(D)^{\oplus (p+1)}$ and $\psi_{g_\varepsilon,\gamma}:=\gamma|z|^2+\log|g_\varepsilon|^2$.
Denote  $B_{a,\varepsilon}:=B_{a\psi_{g_\varepsilon,\gamma},\omega_{g_\varepsilon,\gamma}}=[a i\ddbar\psi_{g_\varepsilon,\gamma},\Lambda_{\omega_{g_\varepsilon,\gamma}}]$. Since  $f$ is an $(n,1)$-form, we have $$B_{a,\varepsilon}^{-1}f=B_{a}^{-1}f=a^{-1}f.$$
Notice that
  \begin{align*}
   \omega_{g_\varepsilon,\gamma}:=& i\partial\dbar \psi_{g_\varepsilon,\gamma}\\
   =&i\gamma\partial\dbar |z|^2+ i\partial\dbar \log(e^{\log |g|^2}+\varepsilon^2)\\
   =&i\gamma\partial\dbar |z|^2+\frac{e^{\log |g|^2}}{e^{\log |g|^2}+\varepsilon^2} i\partial\dbar \log|g|^2+\frac{\varepsilon^2e^{\log |g|^2}}{(e^{\log |g|^2}+\varepsilon^2)^2}i\partial \log|g|^2\wedge\dbar \log|g|^2\\
   \ge&i\gamma\partial\dbar |z|^2+\frac{e^{\log |g|^2}}{e^{\log |g|^2}+\varepsilon^2} i\partial\dbar \log|g|^2\\
  \ge& \frac{ |g|^2}{ |g|^2+\varepsilon^2}\omega_{g,\gamma},
  \end{align*}
  then \begin{align*}
        & \int_{D}^{}\langle B_{a,\varepsilon}^{-1}f,f\rangle_{\omega_
         {g_\varepsilon,\gamma}}e^{-\varphi-a\psi_{g_\varepsilon,\gamma}}
         dV_{\omega_{g_\varepsilon,\gamma}}\\=&a^{-1}\int_{D'}^{}|f|^2_{\omega_
         {g_\varepsilon,\gamma}}e^{-\varphi-a\psi_{g_\varepsilon,\gamma}}
         dV_{\omega_{g_\varepsilon,\gamma}}\\
         \le& a^{-1}\int_{D'}^{}\frac{ |g|^2+\varepsilon^2}{ |g|^2}|f|^2_{\omega_{g,\gamma}}e^{-\varphi-a\psi_{g_\varepsilon,\gamma}}
         dV_{\omega_{g,\gamma}}\\
         \le& a^{-1}\int_{D'}^{}| f|^2_{\omega_{g,\gamma}}e^{-\varphi-a\psi_{g,\gamma}}
         dV_{\omega_{g,\gamma}}\\
         =&\int_{D'}^{}\langle B_{a}^{-1}f,f\rangle_{\omega_{g,\gamma}}e^{-\varphi-a\psi_{g,\gamma}}
         dV_{\omega_{g,\gamma}}.
       \end{align*}

   Since $(D,\varphi)$ is $L^2$-optimal,  there is an $(n,0)$-form $u_\varepsilon$ such that $\dbar u_\varepsilon=f$ on $D$ and
  \begin{align*}
    \int_{D}^{}|u_\varepsilon|^2_{\omega_{g_\varepsilon,\gamma}}
    e^{-\varphi-\psi_{g_\varepsilon,\gamma}}dV_{\omega_{g_\varepsilon,\gamma}}
    \le&\int_{D}^{}\langle B_{a,\varepsilon}^{-1}f,f\rangle_
    {\omega_{g_\varepsilon,\gamma}}
    e^{-\varphi-\psi_{g_\varepsilon,\gamma}}dV_{\omega_{g_\varepsilon,\gamma}}\\
    \le& \int_{D'}^{}\langle B_{a}^{-1}f,f\rangle_{\omega_{g}}
    e^{-\varphi-a\psi_{g,\gamma}}
         dV_{\omega_{g,\gamma}}.
  \end{align*}
  Notice that $u_\varepsilon$ is an  $(n,0)$-form, hence $|u_\varepsilon|^2_{\omega}dV_{\omega}$ is independent of the choice of the metric $\omega$ on the manifold. In addition, $\psi_{g_\varepsilon,\gamma}$ decreasingly converges to $\psi_{g,\gamma}$ as $\varepsilon\to0$,
  then by the Banach-Alaoglu-Bourbaki theorem, we can take a sequence $\varepsilon_j\to0$ such that $u_{\varepsilon_j}$ is weakly $L^2$ convergent to a limit $u$ on $D$. Then it follows from the weakly closedness of $\dbar$ that $\dbar u=f$ on $D$. And by Fatou's lemma and the dominated convergence theorem, we have
  \begin{align*}
     \int_{D'}^{}|u|^2_{\omega_{g,\gamma}}e^{-\varphi-\gamma|z|^2-a\log(|{g}|^2+\delta)}
     dV_{\omega_{g,\gamma}}
     &\le \liminf_{\varepsilon_j\to 0} \int_{D'} |u_{\varepsilon_j}|^2_{\omega_{g,\gamma}}e^{-\varphi-\gamma|z|^2-a\log(|{g}|^2+\delta)}dV_
     {\omega_{g,\gamma}}\\
     &= \liminf_{\varepsilon_j\to 0} \int_{D'} |u_{\varepsilon_j}|^2_{\omega_{g_{\varepsilon_j},\gamma}}e^{-\varphi-\gamma|z|^2-a\log(|{g}|^2+\delta)}
     dV_{\omega_{g_{\varepsilon_j},\gamma}}\\
     &\le\liminf_{\delta>\varepsilon^2_j\to 0} \int_{D'} |u_{\varepsilon_j}|^2_{\omega_{g_{\varepsilon_j},\gamma}}e^{-\varphi-\gamma|z|^2-a\log(|{g}|^2
     +\varepsilon_j^2)}dV_{\omega_{g_{\varepsilon_j},\gamma}}\\
    &\le\int_{D'}^{}\langle B_{a}^{-1}f,f\rangle_{\omega_{g,\gamma}}e^{-\varphi-a\psi_{g,\gamma}}
         dV_{\omega_{g,\gamma}},
  \end{align*}
  where the equality is owing to that $u$ is an $(n,0)$-form.
  Let $\delta\to 0$, it follows from the monotone convergence theorem that
  \begin{align*}
    \int_{D'}^{}|u|^2_{\omega_{g,\gamma}}e^{-\varphi-a\psi_{g,\gamma}}dV_{\omega_{g,\gamma}}\le\int_{D'}^{}\langle B_{a}^{-1}f,f\rangle_{\omega_{g,\gamma}}
    e^{-\varphi-a\psi_{g,\gamma}}dV_{\omega_{g,\gamma}}.
  \end{align*}
\end{proof}

Now we can deduce the desired estimate from the above lemma.

\begin{lem}\label{lem:weak-basic-inequality'}
 Let $D$ be a domain in $\mc^n$  and $\varphi$  an upper semi-continuous function on $D$.
  	Assume that $(D,\varphi)$ is $L^2$-optimal,
	then on $(D';\omega_{g,\gamma},\varphi+a\psi_{g,\gamma})$, we have
	\begin{equation*}
		a\| \alpha\|^2_{\omega_{g,\gamma},\varphi+a\psi_{g,\gamma}}=\iinner{ B_{a} \alpha,\alpha }_{\omega_{g,\gamma},\varphi+a\psi_{g,\gamma}}
			\le 	
		\|\dbar^{*}_{\varphi+a\psi_{g,\gamma}} \alpha \|^2_{\omega_{g,\gamma},\varphi+a\psi_{g,\gamma}}
	\end{equation*}
	holds for $a>0$ and any $(n,1)$-form
$\alpha \in \Dom(\dbar^*_{\varphi+a\psi_{g,\gamma}}) \cap \Ker(\dbar)$.

\end{lem}
\begin{proof}
	
For any $\alpha \in  \Dom(\dbar^*_{\varphi+a\psi_{g,\gamma}})
	\cap \Ker (\dbar)$, we have $ B_{a}\alpha=a\alpha\in \Ker(\dbar)$.
Then by Lemma \ref{bbb}, the assumption that  $(D,\varphi)$ is $L^2$-optimal implies that  there exists $u \in L^2_{n,0}(D';\omega_{g,\gamma},\varphi+a\psi_{g,\gamma})$ such that $\dbar u=B_{a}\alpha$ with
\begin{align*}
  \|u  \|_{\omega_{g,\gamma},\varphi+a\psi_{g,\gamma}}^{2}\le &  \iinner{ B_{a} ^{-1} B_{a}\alpha,B_{a}\alpha }_{\omega_{g,\gamma},\varphi+a\psi_{g,\gamma}} \\
  =& \iinner{ B_{a} \alpha,\alpha }_{\omega_{g,\gamma},\varphi+a\psi_{g,\gamma}}.
\end{align*}
Then
\begin{align}
	&|\iinner{  B_{a}\alpha,\alpha }_{\omega_{g,\gamma},\varphi+a\psi_{g,\gamma}} |^{2}\nonumber\\
	=&|\iinner{ \dbar u, \alpha }_{\omega_{g,\gamma},\varphi+a\psi_{g,\gamma}}|^{2}\nonumber\\
	=&|\iinner{ u,\dbar^{*}_{\varphi+a\psi_{g,\gamma}} \alpha }_{\omega_{g,\gamma},\varphi+a\psi_{g,\gamma}}|^{2}\nonumber\\
	\le&  \|u  \|_{\omega_{g,\gamma},\varphi+a\psi_{g,\gamma}}^{2} \cdot \|\dbar^{*}_{\omega_{g,\gamma},\varphi+a\psi_{g,\gamma}} \alpha \|_{\omega_{g,\gamma},\varphi+a\psi_{g,\gamma}}^{2} \nonumber \\
	\le&   \iinner{ B_{a} \alpha,\alpha }_{\omega_{g,\gamma},\varphi+a\psi_{g,\gamma}}\cdot
	\|\dbar^{*}_{\varphi+a\psi_{g,\gamma}} \alpha \|^2_{\omega_{g,\gamma},\varphi+a\psi_{g,\gamma}}.\nonumber
\end{align}

Thus
$$
\iinner{ B_{a} \alpha,\alpha}_{\omega_{g,\gamma},\varphi+a\psi_{g,\gamma}}
\le 	
\|\dbar^{*}_{\varphi+a\psi_{g,\gamma}} \alpha \|^2_{\omega_{g,\gamma},\varphi+a\psi_{g,\gamma}}.
$$
\end{proof}

\subsection{A variant of Skoda's basic inequality}

In \cite{Skoda}, Skoda obtained the following inequality. It can be viewed as a slight modification of  \cite[Equation (2.12)]{Skoda}, obtained by taking $\beta = \alpha q$ with $q=\min\{n,p-1\}$, and its proof is included in the proof of \cite[Proposition 2]{Skoda}.

\begin{lem}[Skoda's basic inequality \cite{Skoda}]\label{for-skoda}
For any $v=(v_1,\cdots,v_p)$,
where all $v_j=\sum_{k=1}^{n}v_{jk} d\bar z_k $ are smooth $(0,1)$-forms,
we have
  $$
  \min\{n,p-1\}\sum_{j,k,l}^{}\frac{\partial^2\log|g|^2}{\partial z_k\partial\bar z_l}v_{jk}\bar v_{jl}\ge |g|^{2}\left|\sum_{j,k}\frac{\partial(g_je^{-\log|g|^2})}{\partial z_k}v_{jk}\right|^2.$$
\end{lem}

In order to estimate the last term in \eqref{eq:ineq-functional-1}: $2\text{Re}\langle T_2(T_1^* u), v\rangle_{\omega_{g,\gamma},\varphi+a\psi_{g,\gamma}}$, we need to reformulate Skoda's basic inequality as follows.

\begin{lem}\label{for-skoda'}For $v=(v_1,\cdots,v_p)$, where all $v_j=\sum_{k=1}^{n}v_{jk}d\bar z_k$ are $(0,1)$-forms, we have
  $$|g|^2\sum_j\langle \dbar\left(\bar g_je^{-\log|g|^2}\right), v_j\rangle_{\omega_{g,\gamma}}\le
\min\{n,p-1\}|v|^2_{\omega_{g,\gamma}}.$$
\end{lem}
\begin{proof}
Since  $A:=\left(\frac{\partial^2\psi_{g,\gamma}}{\partial z_k\partial\bar z_l}\right)_{kl}=\gamma I_n+\left(\frac{\partial^2\log|g|^2}{\partial z_k\partial\bar z_l}\right)_{kl}$ is Hermitian positive,
set $A^{-1}=(A^{lk})$,
we can take
  $\tilde v_j=A^{-T}v_j=\sum_{k,l=1}^n A^{kl}v_{jk} d\bar{z}_l$, $1\le j\le p$ and $\tilde v=(\tilde v_1,\cdots,\tilde v_p)$. Then
   \begin{align*}
    &|g|^2\left|\sum_j\langle \dbar\left(\bar g_je^{-\log|g|^2}\right), v_j\rangle_{\omega_{g,\gamma}}\right|^2\\
    =&|g|^2\left|\sum_j\langle \dbar\left(\bar g_je^{-\log|g|^2}\right), A^{T}\tilde v_j\rangle_{\omega_{g,\gamma}}\right|^2\\
    =&|g|^{2}\left|\sum_{j,k}\frac{\partial(g_je^{-\log|g|^2})}{\partial z_k}\tilde v_{jk}\right|^2\\
    \le &\min\{n,p-1\}\sum_{j,k,l}^{}\frac{\partial^2(\log|g|^2)}{\partial z_k\partial\bar z_l}\tilde v_{jk}\bar{\tilde v}_{jl}\\
    \le &\min\{n,p-1\} \sum_{j,k,l} A_{kl} \tilde v_{jk}\bar{\tilde v}_{jl}\\
    =&\min\{n,p-1\} \sum_{j,k,l} A_{kl}  \left(\sum_{\alpha} A^{\alpha k} v_{j\alpha} \right) \overline{\left(\sum_{\beta} A^{\beta k} v_{j\beta} \right)}   \\
    =&\min\{n,p-1\} \sum_{j,\alpha,\beta} A^{\alpha \beta}v_{j\alpha} \bar{v}_{j\beta}  \\
   =&\min\{n,p-1\} |v|^2_{\omega_{g,\gamma}},
  \end{align*}
  where the first inequality is due to Lemma \ref{for-skoda} and the second inequality is owing to $\left(\frac{\partial^2\log|g|^2}{\partial z_k\partial\bar z_l}\right)_{kl}\le A$.
  \end{proof}

Finally, combining  \eqref{eq:ineq-functional-1}, Lemma \ref{lem:weak-basic-inequality'} and Lemma \ref{for-skoda'}, we obtain the following basic estimate for the $L^2$ division theorem for $L^2$-optimal pairs.

\begin{lem}\label{lem:skoda-basic-inequality'} With notations as above, we have
\begin{align*}
  &  \| T_1^* u+ T_2^ * v\|_{\omega_{g,\gamma},\varphi+a\psi_{g,\gamma}}^2\\
  \ge & (1-\frac{1}{b}) \|u\|^2_{\omega_{g,\gamma},\varphi+a\psi_{g,\gamma}+\log|g|^2} + (a-b\min\{n,p-1\})\|v\|^2_{\omega_{g,\gamma},\varphi+a\psi_{g,\gamma}}
\end{align*}
holds for any constant $b>1$ and any $u \in \Ker (\dbar) \cap \Dom(T_1^*),
v \in \Dom(T_2^*)\cap(\Ker(T_2^*))^\bot$.
\end{lem}
\begin{proof}
In local coordinates,
we write $u=u_0dz$, $v=(v_1,\cdots,v_p)$, where $dz:=dz_1\wedge\cdots\wedge dz_n$ and all $v_j=\sum_{k=1}^{n}v_{jk}dz\wedge d\bar z_k$ are $(n,1)$-forms, and $$\omega_{g,\gamma}=i\sum_{k,l=1}^{n}\frac{\partial^2\psi_{g,\gamma}}{\partial z_k\partial\bar z_l}dz_k\wedge d\bar z_l.$$ We denote $A:=\left(\frac{\partial^2\psi_{g,\gamma}}{\partial z_k\partial\bar z_l}\right)_{kl}$ and $A^{-1}=(A^{lk})$,
then
\begin{equation*}
  \begin{aligned}
  & \left|2\sum_{j=1}^{p}\text{Re}\left \langle u\dbar\left(\bar g_je^{-\log|g|^2}\right), v_j\right\rangle_{\omega_{g,\gamma}}\right|\\
=& \left|2\sum_{j=1}^{p}\text{Re}\left  \langle u_0\sum_{k=1}^n\frac{\partial(\bar g_je^{-\log|g|^2})}{\partial\bar z_k}dz\wedge d\bar z_k, \sum_{l=1}^nv_{jl}dz\wedge d\bar z_l\right\rangle_{\omega_{g,\gamma}}\right|\\
   =& \left|2\sum_{j=1}^{p}\text{Re}\left(  |dz|^2_{\omega_{g,\gamma}}\cdot u_0\left(\sum_{k,l=1}^nA^{kl}\frac{\partial(\bar g_je^{-\log|g|^2})}{\partial\bar z_k}\bar v_{jl}\right)\right)\right|\\
    \le & \frac{1}{b|g|^2}|dz|^2_{\omega_{g,\gamma}}\cdot |u_0|^2+b|g|^2 |dz|^2_{\omega_{g,\gamma}}\cdot\left| \sum_{j=1}^{p}\sum_{k,l=1}^n A^{kl}\frac{\partial(\bar g_je^{-\log|g|^2})}{\partial\bar z_k}\bar v_{jl}\right|^2\\ =&  \frac{1}{b|g|^2}|u|^2_{\omega_{g,\gamma}}+ b|g|^2|dz|^2_{\omega_{g,\gamma}}\cdot\left| \sum_{j=1}^{p}\langle\bar\partial\left( g_je^{-\log|g|^2}\right), \sum_{l=1}^{n}v_{jl}d\bar z_l\rangle_{\omega_{g,\gamma}}\right|^2
  \\
    \le & \frac{1}{b|g|^2}|u|^2_{\omega_{g,\gamma}}+ b\min\{n,p-1\}|dz|^2_{\omega_{g,\gamma}}\cdot \sum_{j=1}^{p}\left|\sum_{l=1}^{n}v_{jl}d\bar z_l\right|^2_{\omega_{g,\gamma}}\\
    =& \frac{1}{b|g|^2}|u|^2_{\omega_{g,\gamma}}+ b\min\{n,p-1\} \sum_{j=1}^{p}\left|\sum_{l=1}^{n}v_{jl}dz\wedge d\bar z_l\right|^2_{\omega_{g,\gamma}}\\
    =& \frac{1}{b|g|^2}|u|^2_{\omega_{g,\gamma}}+ b \min\{n,p-1\} |v|^2_{\omega_{g,\gamma}},
  \end{aligned}
  \end{equation*}
  where the first inequality is due to the Cauchy-Schwarz inequality and the second inequality is due to Lemma \ref{for-skoda'}.

  Since $u$ is holomorphic, we get that
  \begin{align*}
    &\left|2\text{Re}\iinner{ T_2(T_1^* u), v}_{\omega_{g,\gamma},\varphi+a\psi_{g,\gamma}}\right|\\=& \left|2\sum_{j=1}^{p}\text{Re}\iinner{ u\dbar\left(\bar g_je^{-\log|g|^2}\right), v_j}_{\omega_{g,\gamma},\varphi+a\psi_{g,\gamma}}\right|\\
    =& \left|2\sum_{j=1}^{p}\text{Re}\int_{D'}\langle u\dbar\left(\bar g_je^{-\log|g|^2}\right), v_j\rangle_{\omega_{g,\gamma}}e^{-\varphi-a\psi_{g,\gamma}}dV_{\omega_{g,\gamma}}\right|\\
    \le & \int_{D'}\frac{1}{b|g|^2}|u|^2_{\omega_{g,\gamma}}e^{-\varphi-a\psi_{g,\gamma}}
    dV_{\omega_{g,\gamma}}+\int_{D'} b \min\{n,p-1\} |v|^2_{\omega_{g,\gamma}}
    e^{-\varphi-a\psi_{g,\gamma}}dV_{\omega_{g,\gamma}}\\
    =& \frac{1}{b}\|u\|^2_{\omega_{g,\gamma},\varphi+a\psi_{g,\gamma}+\log|g|^2}+b\min\{n,p-1\}
    \|v\|^2_{\omega_{g,\gamma},\varphi+a\psi_{g,\gamma}}.
  \end{align*}

By Lemma \ref{lem:weak-basic-inequality'},
we have
$$
 \| T^*_2 v \|^2 _{\omega_{g,\gamma},\varphi+a\psi_{g,\gamma}} \ge  a \|v\|^2_{\omega_{g,\gamma},\varphi+a\psi_{g,\gamma}}.
$$
Therefore,
\begin{eqnarray*}
  & &  \| T_1^* u+ T_2^ * v\|_{\omega_{g,\gamma},\varphi+a\psi_{g,\gamma}}^2\\
  &\ge & \| T^*_2 v \|^2 _{\omega_{g,\gamma},\varphi+a\psi_{g,\gamma}} +(1-\frac{1}{b}) \|u\|^2_{\omega_{g,\gamma},\varphi+a\psi_{g,\gamma}+\log|g|^2} -b\min\{n,p-1\}\|v\|^2_{\omega_{g,\gamma},\varphi+a\psi_{g,\gamma}} \\
  &\ge &(1-\frac{1}{b}) \|u\|^2_{\omega_{g,\gamma},\varphi+a\psi_{g,\gamma}+\log|g|^2} +(a-b\min\{n,p-1\})\|v\|^2_{\omega_{g,\gamma},\varphi+a\psi_{g,\gamma}}.
\end{eqnarray*}

\end{proof}

\section{Skoda's $L^2$ division theorem for $L^2$-optimal pairs}\label{sec:prove-main}

Firstly, we give the proof of Skoda's $L^2$ division theorem for $L^2$-optimal pairs.

\begin{thm}[Theorem \ref{thm:skoda-division'}]
Let $D$ be a domain in $\mc^n$, $\varphi$ an upper semi-continuous function on $D$
  and $g\in\calO(D)^{\oplus p}$.
	Assume that $(D,\varphi)$ is $L^2$-optimal. Set $\varepsilon>0$ and $m=\min\{n,p-1\}$. Then for any holomorphic  $(n,0)$-form $f$  with $$\int_Di^{n^2} f\wedge\bar fe^{-\varphi-(1+m+\varepsilon)\log|g|^2}<+\infty,$$ there exist holomorphic  $(n,0)$-forms $(h_1,\cdots,h_p)$ on $D$
such that
$\sum_{j=1}^{p} h_j g_j=f$
and
\begin{equation*}
  \sum_{j=1}^{p}\int_Di^{n^2}h_j\wedge\bar h_je^{-\varphi-(m+\varepsilon)\log|g|^{2}}
  \le(1+\frac{m}{\varepsilon})\int_Di^{n^2}f\wedge\bar fe^{-\varphi-(1+m+\varepsilon)\log|g|^2}.
\end{equation*}
\end{thm}

\begin{proof}

If $m=0$, i.e. $p=1$, we take $h=\frac{f}{g}$, then $h$ is holomorphic on $D\setminus\{g=0\}$ and
\begin{equation*}
 \int_{D}i^{n^2}h\wedge\bar he^{-\varphi-\varepsilon\log|g|^{2}}=
  \int_{D}i^{n^2}f\wedge\bar fe^{-\varphi-(1+\varepsilon)\log|g|^2}<+\infty.
\end{equation*}
Then by Lemma \ref{lem L2negligible},  $h$ is holomorphic on $D$.

If $m\ge 1$,
taking $a=m+\varepsilon=bm$ in Lemma $\ref{lem:skoda-basic-inequality'}$, together with Lemma \ref{lem:weak-basic-inequality'}, we obtain
\begin{align*}
    \| T_1^* u+ T_2^ * v\|_{\omega_{g,\gamma},\varphi+a\psi_{g,\gamma}}^2
  \ge \frac{\varepsilon}{m+\varepsilon} \|u\|^2_{\omega_{g,\gamma},\varphi+a\psi_{g,\gamma}+\log|g|^2}
\end{align*}
for any $u \in \Ker (\dbar) \cap \Dom(T_1^*),
v \in \Dom(T_2^*)\cap(\Ker(T_2^*))^\bot$.
Now for any holomorphic  $(n,0)$-form $f$  with $$\int_Di^{n^2}f\wedge\bar fe^{-\varphi-(1+m+\varepsilon)\log|g|^2}<+\infty,$$
we have
\begin{align*}
  \int_{D'}|f|^2_{\omega_{g,\gamma}}e^{-\varphi-(m+\varepsilon)\psi_{g,\gamma}-\log|g|^2}
  dV_{\omega_{g,\gamma}}
  =&\int_{D'}i^{n^2}f\wedge\bar fe^{-\varphi-(1+m+\varepsilon)\log|g|^2-\gamma(m+\varepsilon)|z|^2}\\
  \le&\int_{D'}i^{n^2}f\wedge\bar fe^{-\varphi-(1+m+\varepsilon)\log|g|^2}<+\infty.
\end{align*}
Then by Lemma \ref{lem:functional-lemma-Skoda}, there exist holomorphic  $(n,0)$-forms $h^{(\gamma)}=(h_1^{(\gamma)},\cdots,h_p^{(\gamma)})$ on $D'$
such that
$\sum_{j=1}^{p} h_j^{(\gamma)} g_j=f$ on $D'$
and
\begin{equation*}
\int_{D'}|h^{(\gamma)}|^2_{\omega_{g,\gamma}}
e^{-\varphi-(m+\varepsilon)\psi_{g,\gamma}}dV_{\omega_{g,\gamma}}
  \le(1+\frac{m}{\varepsilon})
  \int_{D'}|f|^2_{\omega_{g,\gamma}}e^{-\varphi-(m+\varepsilon)\psi_{g,\gamma}-\log|g|^2}
  dV_{\omega_{g,\gamma}}.
\end{equation*}
Since all $h^{(\gamma)}_j$ are $(n,0)$-forms, hence $|h^{(\gamma)}|^2_{\omega}dV_{\omega}$ is independent of the choice of the metric $\omega$ on the manifold. In addition, $\psi_{g,\gamma}$ decreasingly converges to $\log|g|^2$ as $\gamma\to0$,
  then by the Banach-Alaoglu-Bourbaki theorem, we can take a sequence $\gamma_j\to0$ such that $h^{(\gamma_j)}$ is weakly $L^2$ convergent to a limit $h=(h_1,\cdots,h_p)$ on $D'$. Then $h$ is also holomorphic on $D'$ and $\sum_{j=1}^{p} h_j g_j=f$ on $D'$. And by Fatou's lemma and the dominated convergence theorem, we have
\begin{equation*}
  \sum_{j=1}^{p}\int_{D'}i^{n^2}h_j\wedge\bar h_je^{-\varphi-(m+\varepsilon)\log|g|^{2}}
  \le(1+\frac{m}{\varepsilon})\int_Di^{n^2}f\wedge\bar fe^{-\varphi-(1+m+\varepsilon)\log|g|^2}.
\end{equation*}
Finally, by Lemma \ref{lem L2negligible}, all $h_j$ can be extended to $D$.
\end{proof}

Furthermore, we want to generalize the above result on Stein manifolds.
To be precise, let $X$ be a Stein manifold, $D$  a domain in $X$ and $L$ a holomorphic line bundle endowed with a possibly singular metric $e^{-\varphi}$ over $D$.

\begin{defn}[\cite{Liu-Zhang}]\label{def stein l2 optimal}
 We say that  the pair $(D,L,\varphi)$ is \emph{$L^2$-optimal}
  if for any smooth strictly plurisubharmonic function $\phi$ and
  any K\"{a}hler metric $\omega$ on $D$,
  the equation $\dbar u=f$ can be solved on $D$ for any $\dbar$-closed $L$-valued $(n,1)$-form
  $f \in L^2_{n,1}(D,\phi)$
  with the estimate:
  $$
  \int_D|u|^2_{\omega,\varphi} e^{-\phi} dV_\omega
  \le 
  \int_D \inner{B_\phi^{-1}f,f}_{\omega,\varphi} e^{-\phi} dV_{\omega},
  $$
  provided the right-hand side is finite, where $B_\phi:=[i\ddbar \phi, \Lambda_{\omega}]$.

  In particular, we call $D$ is  $\emph{$L^2$-optimal}$ if the pair $(D,D\times\mc,0)$ is \emph{$L^2$-optimal}.
\end{defn}

It is well known that $X$ admits a smooth strictly plurisubharmonic exhaustion function $\eta$.  By replacing $\gamma|z|^2$ with  $\gamma\eta$ and calculating in normal coordinates,  we can obtain the Skoda $L^2$ division theorem for $L^2$-optimal pairs on Stein manifolds similarly but omit the proof.

\begin{thm}[Theorem \ref{thm:skoda-division-line-bundle}]
Let $D$ be a domain in a Stein manifold $(X,\omega)$, $L$ be a holomorphic line bundle endowed with a possibly singular metric $e^{-\varphi}$ over $D$ and $g=(g_{1},\cdots, g_{p}) \in \calO(D)^{\oplus p}$. Assume that the local weight function $\varphi$ is upper semi-continuous on $D$  and the pair $(D,L,\varphi)$ is $L^2$-optimal.
Set $\varepsilon>0 $ and $ m=\min \{n,p-1 \}$.
Then for any $L$-valued $(n,0)$-form $f$ satisfying
$$
\int_D |f|^2_{\omega,\varphi} e^{-(1+m+\varepsilon)\log|g|^2} dV_{\omega}<+\infty,
$$
there exist holomorphic $L$-valued $(n,0)$-forms $(h_1,\cdots,h_p)$ on $D$
such that
$$
\sum_{j=1}^{p} h_j g_j=f
$$
and
\begin{equation*}
  \sum_{j=1}^{p}\int_D |h_j|^2_{\omega,\varphi} e^{-(m+\varepsilon)\log|g|^{2}} dV_{\omega}
  \le(1+\frac{m}{\varepsilon})\int_D |f|^2_{\omega,\varphi}e^{-(1+m+\varepsilon)\log|g|^2}dV_{\omega}.
\end{equation*}
\end{thm}

\section{Applications}
\label{sec:application}
In this section, we present some applications of the Skoda $L^2$ division theorem for $L^2$-optimal pairs.
Firstly, we can remove the assumption of boundedness for the domain when characterizing pseudoconvex domains via $L^2$-optimal conditions.

\begin{thm}[Corollary \ref{cor1}]
  Let $D\subset\mc^n$ be a domain with a null thin complement. Then  $D$ is  $L^2$-optimal if and only if $D$ is  a domain of holomorphy.
\end{thm}

\begin{proof}If $D$ is  a domain of holomorphy, then $D$ is pseudoconvex. It follows from H\"omander's $L^2$ existence (\cite[Chapter VIII-(6.1)]{Demailly}) that $D$ is $L^2$-optimal.

Conversely,
  suppose that  $D$ is not a domain of holomorphy. Then by the definition,  there is a connected open set $U_2$ with $U_2 \cap D\neq\emptyset$ and $U_2 \not\subset D$ and an open subset $U_1\subset U_2\cap D$ such that for every holomorphic function $f\in \calO(D)$, there is an $F\in \calO(U_2)$ satisfying $F=f$ on $U_1$.

  Since $D$ is $L^2$-optimal, $(D,(n+\varepsilon)\log(1+|z|^2))$ is also $L^2$-optimal. Since  $\mathring{\overline {D}}=D$, we have $U_2 \not\subset \overline {D}$. Thus for some $z^0\in U_2 \backslash \overline {D}$, in Theorem \ref{thm:skoda-division'}, we can take $g(z)=z-z^0$.  Noticing that $dz\in L^2_{(n,0)}(D,(n+\varepsilon)(\log(1+|z|^2)+\log|z-z^0|^2))$, we obtain holomorphic $(n,0)$-forms  $h=(h_1,\cdots,h_n)$
with
$\sum_{j=1}^{n}  (z_j-z_j^0)h_j(z)=dz$ on $D$. Notice that all $\frac{h_j}{dz}$ are holomorphic functions on $D$,
 then there exist $H_j\in\calO(U_2), j=1,\cdots,n,$ such that $$H_j\big|_{U_1}=\frac{h_j}{dz}\big|_{U_1}.$$
Then on $U_1$ we have
$$
\sum_{j=1}^{n} H_j(z)(z_j-z^0_j) \equiv1.
$$
Hence by the uniqueness theorem for holomorphic functions, we obtain $\sum_{j=1}^{n} H_j(z)(z_j-z^0_j) \equiv1$
on $U_2$.
In particular, $\sum_{j=1}^{n} H_j(z)(z_j-z^0_j) =1$ at $z^0\in U_2$, which is a contradiction.

\end{proof}

 Moreover, using the distance function, we provide a new characterization of domains of holomorphy via the $L^2$-optimal condition.

\begin{thm}[Corollary \ref{cor2}]
  Let $D$ be a domain in $\mc^n$. Then $D$ is a domain of holomorphy if and only if $(D,-\log d(z,D^c))$  is $L^2$-optimal.
\end{thm}

\begin{proof}
If $D$ is a domain of holomorphy, then $-\log d(z,D^c)$ is plurisubharmonic. Therefore, $(D,-\log d(z,D^c))$  is $L^2$-optimal by H\"ormander's $L^2$ existence.

Conversely, 
given any open connected set $U$ such that $U\cap\partial D\neq\emptyset$, choose $z^0\in U\cap\partial D$. Since $(D,-\log d(z,D^c))$  is $L^2$-optimal and $\log (1+|z|^2)$ is smooth plurisubharmonic, then $(D,-\log d(z,D^c)+(n+\varepsilon)\log (1+|z|^2))$ is also $L^2$-optimal.

         Notice that for $\varepsilon\in(0,\frac{1}{2})$,
         \begin{align*}
           &\int_De^{-(n+\varepsilon)\log (1+|z|^2)-(n+\varepsilon)\log|z-z^0|^2+\log d(z,D^c)}i^{n^2}dz\wedge d\bar z\\
           \le&\int_D \frac{1}{(1+|z|^2)^{n+\varepsilon}}\cdot\frac{1}{|z-z^0|^{2n-1+2\varepsilon}}
         i^{n^2}dz\wedge d\bar z\\
           \le&\int_{\{|z-z^0|\le1\}} \frac{1}{|z-z^0|^{2n-1+2\varepsilon}}i^{n^2}dz\wedge d\bar z
           +\int_{\{|z-z^0|\ge1\}} \frac{1}{(1+|z|^2)^{n+\varepsilon}}i^{n^2}dz\wedge d\bar z
          \\< &+\infty,
         \end{align*}
         where $dz:=dz_1\wedge\cdots\wedge dz_n$.
         In Theorem \ref{thm:skoda-division'}, we take $g=z-z^0\in\calO(D)^{\oplus n}$ and $f=dz$, then there exist holomorphic  $(n,0)$-forms $h=(h_1,\cdots,h_n)$ on $D$
such that
$\sum_{j=1}^{n}  (z_j-z_j^0)h_j(z)=dz$ on $D$.  This shows that at least
one of the functions $\frac{h_j}{dz}$ cannot be analytically continued across the point $z^0\in \partial D$.
By the Cartan-Thullen theorem, $D$ is a domain of holomorphy.
\end{proof}

Finally, we translate Skoda's $L^2$ division theorem for $L^2$-optimal pairs into the following algebraic geometry formulation.

\begin{thm}[Corollary \ref{thm:algebraic}]
  Let $X$ be a Stein manifold and $L$ be a holomorphic line bundle endowed with a possibly singular metric $e^{-\varphi}$ over $X$. Assume that the local weight function $\varphi$ is upper semi-continuous on $X$ and the pair $(X,L,\varphi)$ is $L^2$-optimal. Let $M$ be an effective holomorphic line bundle on $X$ and $g=(g_{1},\cdots, g_{p}) \in (H^0(X,M))^{\oplus p}$.
Let $k  \ge 1$ be an integer and
define
	$$
	\calI_k =\calI (\varphi+(n+k)\log |g|^{2}),
	$$
then we have
\begin{equation*}
	H^0(X,K_X\otimes L\otimes (n+k+1)M\otimes\calI_{k+1})=\sum^{p}_{j=1} g^{j} H^0 (X,K_X\otimes L\otimes (n+k)M\otimes\calI_k).
\end{equation*}
\end{thm}

\begin{proof}
It is clear that the right-hand side is contained in the left-hand side of the equality in the conclusion. For the converse direction, 
let $f\in H^0(X,K_X\otimes L\otimes (n+k+1)M\otimes\calI_{k+1})$.
We may assume $m:=\min\{n,p-1\}=n$ by adding $g^{p+1}=\cdots=g^{n}=0$ if $p \le n$.

When $X$ is Stein, there is a smooth strictly plurisubharmonic exhaustion function $\rho$ on $X$ such that
$$
\int_{X} \frac{|f|^2e^{-\varphi-\rho}}{|g|^{2(n+k+1)}}<+\infty.
$$	
Since $X$ is  Stein, there is a hypersurface $A\subset X$ such that  both $L, F$ are  trivial on $X':=X\setminus A$. Thus we can write
\begin{align*}
   f=& F e_L\otimes e_M^{\otimes(n+k+1)},\\
      g_j=&G_j e_M, \quad \quad 1\le j\le n.
\end{align*}
on $X'$, where $G_j$ are holomorphic functions and 
$F$ is holomorphic $(n,0)$-form, with 
$e_L$ and $e_M$ are holomorphic frames of $L$ and $M$ respectively.
Then we have $$\int_{X'}\frac{|F|^2e^{-\varphi-\rho}}{|G|^{2(n+k+1)}}<+\infty.$$

Since  $(X,L,\varphi)$ is $L^2$-optimal and $\rho$ is smooth plurisubharmonic on $X$, we know that
  $(X,L,\varphi+\rho)$ is also $L^2$-optimal.  For the pair  $(X,L,\varphi+\rho)$, repeat the procedure of the proof of Skoda's $L^2$ division theorem by replacing  $\gamma|z|^2$ with  $\gamma\rho$ and calculating in normal coordinates, then we can find  holomorphic $(n,0)$-forms $H_1,\cdots, H_n$
such that
$$\sum_{j=1}^{n} H_j G_j=F$$
and
$$
  \int_{X} \frac{|H|^2e^{-\varphi-\rho}}{|G|^{2(n+k)}}
  \le(1+\frac{n}{k})\int_{X} \frac{|f|^2e^{-\varphi-\rho}}{|G|^{2(n+k+1)}}<+\infty.$$
Now $h_j=H_je_L\otimes e_M^{\otimes(n+k)}\in H^0 (X',K_X\otimes L\otimes (n+k)M\otimes\calI_k)$ is locally $L^2$ integrable near $A$. It follows from Lemma \ref{lem L2negligible} that $h_j\in H^0 (X,K_X\otimes L\otimes (n+k)M\otimes\calI_k)$ and we have
$$\sum_{j=1}^{n} h_j\otimes g_j=f.$$
We complete the proof.

\end{proof}

\end{document}